\documentclass[11pt,a4paper,twoside,reqno]{amsart}
  \usepackage[all,2cell,arrow]{xy}
  \usepackage{amsfonts}
  \usepackage{amsthm}
  \usepackage{amsmath}
  \usepackage{amssymb}
\usepackage{mathtools}
\usepackage{tensor}
 \numberwithin{equation}{section}
 \usepackage{color}
\usepackage{graphicx}
  \usepackage{hyperref}
\usepackage{enumerate,xspace}

\UseAllTwocells
\CompileMatrices

\definecolor{darkgreen}{rgb}{0,0.45,0} 
\definecolor{lightgrey}{rgb}{0.666666,0.666666,0.666666}

\renewcommand{\epsilon}{\varepsilon}
\renewcommand{\phi}{\varphi}

\newcommand{\ca}{\ensuremath{\mathcal A}\xspace}
\newcommand{\cb}{\ensuremath{\mathcal B}\xspace}
\newcommand{\cc}{\ensuremath{\mathcal C}\xspace}

\newcommand{\cl}{\ensuremath{\mathcal L}\xspace}

\newcommand{\cn}{\ensuremath{\mathcal N}\xspace}

\newcommand{\crr}{\ensuremath{\mathcal R}\xspace}

\newcommand{\ct}{\ensuremath{\mathcal T}\xspace}

\newcommand{\bba}{\ensuremath{\mathbb A}\xspace}
\newcommand{\bbb}{\ensuremath{\mathbb B}\xspace}
\newcommand{\bbc}{\ensuremath{\mathbb C}\xspace}

\newcommand{\bbn}{\ensuremath{\mathbb N}\xspace}

\newcommand{\tmop}{\cn}

\newcommand{\Mat}{\ensuremath{\mathbf{Mat}}\xspace}
\newcommand{\Cat}{\ensuremath{\mathbf{Cat}}\xspace}

\newcommand{\Set}{\ensuremath{\mathbf{Set}}\xspace}

\newcommand{\co}{\ensuremath{^{\textnormal{co}}}}
\newcommand{\op}{\ensuremath{^{\textnormal{op}}}}

\newcommand{\TMult}{\ensuremath{\mathbf{T\textnormal{-}Mult}}\xspace}
\newcommand{\Skew}{\ensuremath{\mathbf{Skew_\ell}}\xspace}
\newcommand{\SkewC}{\ensuremath{\mathbf{SkewCl_\ell}}\xspace}
\newcommand{\SkewMC}{\ensuremath{\mathbf{SkewMCl_\ell}}\xspace}

\newcommand{\RMult}{\ensuremath{\mathbf{R\textnormal{-}Mult}}\xspace}

\newcommand{\Mult}{\ensuremath{\mathbf{Mult}}\xspace}

\newcommand{\TAlgs}{\ensuremath{\mathbf{T}\textnormal{-}\mathbf{Alg}_{s}}\xspace}

\newcommand{\TAlg}{\ensuremath{\mathbf{T}\textnormal{-}\mathbf{Alg}}\xspace}

\newcommand{\LBCAlgc}{\ensuremath{\mathbf{LBC\textnormal{-}Alg}}\xspace}
\newcommand{\LBC}{\textnormal{LBC}}
\newcommand{\nc}{\ensuremath{\mathbf{nColax\textnormal{-}T^{*}\textnormal{-}Alg}}\xspace}
\newcommand{\ncnotdual}{\ensuremath{\mathbf{nColax\textnormal{-}T\textnormal{-}Alg}}\xspace}
\newcommand{\ncl}{\ensuremath{\mathbf{nColax\textnormal{-}L\textnormal{-}Alg}}\xspace}

\newcommand{\myt}{t}
\newcommand{\myl}{\ell}
\newcommand{\mylam}{\lambda}

\DeclareMathOperator{\End}{End}

\newcommand{\two}{\ensuremath{\mathbf{2}\xspace}}

\newcommand{\ox}{\otimes}
\newcommand{\x}{\times}

\def\1c#1{\stackrel{#1}{\to}}

  \newtheorem{proposition}{Proposition}[section]

  \newtheorem{theorem}[proposition]{Theorem}

  \theoremstyle{definition}
  \newtheorem{definition}[proposition]{Definition}
  \newtheorem{notation}[proposition]{Notation}
  \newtheorem{example}[proposition]{Example}
  \newtheorem{examples}[proposition]{Examples}
  \newtheorem{AlternativeViewpoint}{Alternative perspective}

  \theoremstyle{remark}
  \newtheorem{remark}[proposition]{Remark}

  \newcounter{c}
  \renewcommand{\[}{\setcounter{c}{1}$$}
  \newcommand{\etyk}[1]{\vspace{-7.4mm}$$\begin{equation}\Label{#1}
  \addtocounter{c}{1}}
  \renewcommand{\]}{\ifnum \value{c}=1 $$\else \end{equation}\fi}
\begin{document}

 \title{Skew monoidal categories and skew multicategories}

\author{John Bourke, Stephen Lack}
\address{Department of Mathematics, Macquarie University NSW 2109, Australia}
\email{steve.lack@mq.edu.au}

\date{\today}

\begin{abstract}
We describe a perfect correspondence between skew monoidal categories and certain generalised multicategories, called skew multicategories, that arise in nature.
\end{abstract} 
\date\today
\maketitle


\section{Introduction}

In linear algebra one meets the concept of a bilinear map $A\x B\to
D$, and later learns that these are in bijection with linear maps
$A\ox B\to D$, so that one might say that bilinear maps are
``classified'' by the tensor product $A\ox B$. Similarly there are
trilinear maps $A\x B\x C\to D$ and these are classified by the tensor
product $A\ox B\ox C$, which can be constructed out of the binary
construction as either of the (isomorphic) objects  $(A\ox B)\ox
C$ and $A\ox (B\ox C)$. The most important
properties of the tensor product, including coherence, all follow from the universal
property, without the need for an explicit construction of the tensor product.

In many other examples of monoidal structures which are not given by a
categorical (cartesian) product, it is likewise the case that maps out
of a tensor product $A\ox B$ can alternatively be described in terms
of some more primitive notion, corresponding to bilinearity. 

This idea of multilinear morphisms can be abstracted in the notion of
{\em multicategory} \cite{Lambek-multicategories}, and in such a
multicategory one can then ask whether or not there is a suitable corresponding
tensor product, in which case the multicategory is said to be {\em
  representable}.  In this case the corresponding tensor product forms
  part of a monoidal structure, and indeed the notion of a representable
   multicategory is equivalent to that of a monoidal category.
  


There are various intermediate structures between monoidal categories
and multicategories, including the notion of {\em colax monoidal
  category}.  Colax monoidal structure on a category \ca is the
same as lax monoidal structure on $\ca\op$; it involves an $n$-ary
product $\ca^n\to \ca$ for each $n$. The colax part of the structure
 consists of  maps 
\[ a_1a_2\ldots a_n \to (a_1\ldots a_i)(a_{i+1}\ldots
  a_j)\ldots(a_{k+1}\ldots a_n) \]
going {\em into} a product of products, from the expanded product;
these are of course subject to various conditions. The case $n=0$ of
the product amounts to an object $i\in \ca$, and part of the colax
structure involves maps $a\to ia$ and $a\to ai$. 
 For the connection between colax monoidal categories and
multicategories, see \cite{DayStreet-LaxMonoids} or \cite{Aguiar2017Monads}.
For applications of lax monoidal structure to higher categories, see 
\cite{Leinster-book} and \cite{BataninWeber-EnrichedI}. 

Recently, the notion of {\em skew monoidal category} has received a
lot of attention, due to Szlach\'anyi's brilliant insight
\cite{Szlachanyi-skew} that they
can be used to describe bialgebroids. A skew monoidal category is a
category $\ca$ equipped with a binary product $\ca^2\to \ca$ and a unit
object $i$, together with natural transformations 
\[ \xymatrix @R0pc { 
(ab)c \ar[r]^{\alpha} & a(bc) \\
ia \ar[r]^{\lambda} & a \\
a \ar[r]^{\rho} & ai }\]
subject to five axioms corresponding to Mac~Lane's five axioms for a
monoidal category \cite{Maclane-monoidal}.   The theory of skew monoidal categories 
has been developed in a series of papers \cite{skew,skewcoherence,skewEH,skew-reflection,mw} 
by the second named author in collaboration with Street.
%

In \cite{bourko-skew} the first named author described examples of skew 
monoidal categories arising from a rather different source -- 2-category theory --
and used these to give efficient constructions
of various more complex monoidal bicategories.  In these 2-categorical examples
 the skew structures were shown to arise from multicategories 
with two types of multimorphism -- one stricter, one weaker. 

The main goal of the present paper is to describe a perfect correspondence
between skew monoidal categories and a kind of multicategory with two
types of multimorphism; unsurprisingly, we call these \emph{skew multicategories}.

The starting point is the fact that, for a skew monoidal category $\ca$,
the functor $\ca\to \ca$ given by tensoring on the left with the unit
object $i$ is a comonad (with counit $\lambda$). The category $\ca$
therefore comes equipped with a notion of ``weak morphism'' from $a$
to $b$, consisting of a morphism in the usual sense from $ia$ to
$b$. These can be composed as in the Kleisli category of the comonad.
But there are also ``weak multimorphisms'', classified by products
such as $(ia)b$ or $((ia)b)c$, as well as a stricter sort of multimap
classified by products such $ab$ or $(ab)c$.

It turns out that the resulting notion of skew multicategory is
controlled by a particular (non-symmetric) \Cat-enriched operad $\crr$,
and it is convenient to develop the general notion of
$\ct$-multicategory for a \Cat-operad $\ct$; then skew multicategories
are $\crr$-multicategories, and ordinary multicategories are
\tmop-multicategories, where \tmop is the terminal \Cat-operad.  
These $\ct$-multicategories can be described using existing
notions of generalised multicategory, like that of \cite{CruttwellShulman}.
The various approaches are described in more detail in Section~\ref{sect:Tmult}.

Just as for ordinary monoidal categories and multicategories, various
aspects of skew monoidal categories can most easily be seen from the
point of view of skew multicategories. In a future paper
\cite{skewsymmetric}, we shall study braidings and symmetries for skew
monoidal categories. If one thinks of these as extra structure
borne by the tensor product, one might be led to write down an overly 
simplistic definition, but from the multicategorical point of view
things become far clearer. 

 We now outline the structure of the current paper.  
After a brief review of multicategories and operads, we begin with the
notion of $\ct$-multicategory in Section~\ref{sect:Tmult} and discuss
representability in this context.  In Section~\ref{sect:Skewmult} we
define skew multicategories as $\crr$-multicategories for a
$\Cat$-operad $\crr$, giving examples and defining left representable
skew multicategories.  
 In Section~\ref{sect:Talg} we define  colax
$\ct$-algebras, relate  them to $\ct$-multicategories,
 and identify  
those corresponding to left representable skew multicategories.  

Using  this, as well as results from our companion paper
\cite{Fsk2},  we describe in Section~\ref{sect:Skewmon} 
the perfect correspondence between skew monoidal categories and left
representable skew multicategories.   In Section~\ref{sect:Skewmon} we also
describe variants of this correspondence, dealing with closed skew monoidal categories and skew closed categories.

\subsection*{Acknowledgements}

Both authors acknowledge with gratitude the support of an Australian Research Council
Discovery Grant DP130101969; Lack further acknowledges the support of a Future Fellowship FT110100385.

\section{Review of multicategories and operads}\label{sect:review}

In this section we briefly review the definitions, establishing our
terminology along the way. 

In this paper our operads will always be of the ``plain'' variety,
without actions of the symmetric groups. On the other hand, they will
usually be enriched over \Cat. Such plain \Cat-enriched operads can be
equivalently be described as clubs over $\bbn$ \cite{Kelly:clubs}.

\subsection{Multicategories}

A {\em multicategory} \bba consists of:
\begin{itemize}
\item a collection of objects
\item for each (possibly empty) list $a_1,\ldots,a_n$ of objects and each object
  $b$, a set $\bba(a_1,\ldots,a_n;b)$; sometimes we write
  $\overline{a}$ for the list, and then
  $\bba(\overline{a};b)$ for the set
\item for each object $a$, an element $1_a\in \bba(a;a)$
\item substitution operations
\[ \bba(b_1,\ldots,b_n;c)\x \prod\limits^n_{i=1}
\bba(\overline{a}_i;b_i) \to 
\bba(\overline{a}_1,\ldots,\overline{a}_n;c)
\]
where each $\overline{a}_i$ is itself a list $a_{i1},\ldots,a_{ik_i}$,
and where in the codomain these lists have been concatenated to obtain
a list $a_{11},\ldots,a_{nk_n}$. 
\end{itemize}
The notation for substitution is 
$(g,f_1,\ldots,f_n) \mapsto g(f_1,\ldots,f_n)$.
We require the evident associativity conditions as well as
the identity laws $1_c(g)=g=g(1_{b_1},\ldots,1_{b_n})$.

\begin{remark}
Given $g \in \bba(b_1,\ldots,b_n;c)$ and $f\colon\overline{a} \to b_{i}$ we sometimes write $g \circ_{i} f$ for the multimap $g(1,\ldots,1,f,1,\ldots,1)$ 
which captures \emph{substitution in position $i$.} In the present paper we will use $g \circ_{i} f$ only to simplify notation but note that the notion of multicategory can be formulated with these operations as the primitive ones (see \cite{Markl-OperadsPROPs} for the operad case). 


\end{remark}

In a {\em \Cat-enriched multicategory} each
$\bba(a_1,\ldots,a_n;b)$ is now a category, and each of the
substitution operations a functor. The equations are required to hold
strictly.

\subsection{Operads}

When a multicategory has only one object, there is no need to keep
track of the objects $a_1,\ldots,a_n,b$ in $\bba(a_1,\ldots,a_n;b)$
and so one may simply write $\bba_n$, except that we typically use
names like $\ct$ rather than $\bba$ for an operad, and so we would
write $\ct_n$ rather than $\bba_n$.  There is then a unit $e\in \ct_1$;
and substitution operations $\ct_n\x \ct_{k_1}\x\ldots\x \ct_{k_n}\to
\ct_{k_1+\ldots+k_n}$.

We are generally interested in the \Cat-enriched case, in which we
shall speak of a \Cat-operad.   Each $\Cat$-operad $\ct$ has a dual $\ct^*$ with
$\ct^*_n=\ct\op_n$, and with
the corresponding multiplication and unit. (From the point of view of
\Cat-enriched multicategories, one might think of this as ``$\ct\co$'',
obtained by reversing the 2-cells but not the 1-cells.)

\section{$\ct$-multicategories}\label{sect:Tmult}

{ Let $\ct$ be a fixed \Cat-operad; recall that in this paper our operads
do not involve actions of the symmetric groups. }

\begin{definition}
A {\em $\ct$-multicategory} \bba consists of a set $A$ of objects,
together with 
\begin{itemize}
\item for each list $a_1,\ldots,a_n\in A$ and each $b\in A$, a functor 
\[ \bba(a_1,\ldots,a_n;b)\colon \ct_n\to\Set \]
whose value at an object $x\in \ct_n$ we write as
$\bba_x(a_1,\ldots,a_n;b)$
or sometimes $\bba_x(\overline{a};b)$, where $\overline{a}$ stands for
the list $a_1,\ldots,a_n$;
\item for each $a\in A$ an element $1_a\in \bba_e(a;a)$ called the
  identity;
\item substitution maps
\[ \xymatrix @R0pc {
\bba_{x}(b_1,\ldots,b_n;c)\x \prod\limits^n_{i=1} 
\bba_{x_{i}}(\overline{a}_i;b_i) \ar[r] & \bba_{x(x_{1},\ldots,x_{n})}(\overline{a}_1,\ldots,\overline{a}_n;c) \\
(g,f_{1},\ldots f_{n}) \ar@{|->}[r] & g(f_{1},\ldots,f_{n})
 } \]
natural in $x,x_1,\ldots,x_n$
\end{itemize}
satisfying the associativity and identity axioms which are the
natural ``$\ct$-typed'' analogues of those for ordinary multicategories. 
\end{definition}

\begin{example}
  If $\ct$ is the terminal \Cat-operad \tmop with $\tmop_n=1$ for all $n$, a
  $\ct$-multicategory is just an ordinary multicategory. 
\end{example}

\begin{remark}
In the $\ct$-multicategory context we also  sometimes write $g \circ_{i} f$ for the multimap $g(1,\ldots,1,f,1,\ldots,1)$.  And as in the ordinary setting, $\ct$-multicategories admit a formulation taking such operations as primitive.

\end{remark}

Just as for ordinary multicategories, every
$\ct$-multicategory $\bba$ has an associated category $\ca$ with the same objects,
with homs given  by $\ca(a,b)=\bba_e(a;b)$, and with composition given by substitution.

\begin{proposition}\label{prop:categoryA}
For each $n$, the functor
\[  \bba_{-}(-;-)\colon\ct_n\x (A^n)\op\x A\to \Set \]
extends to a functor 
\[  \bba_{-}(-;-)\colon\ct_n\x (\ca^n)\op\x \ca\to \Set. \]
These extensions are uniquely determined by the following properties:
\begin{enumerate}
\item $\bba_{e}(-,-)=\ca(-,-)\colon\ca^{op} \times \ca \to \Set$;
\item  the substitution maps
\[ \bba_{x}(b_1,\ldots,b_n;c)\x \prod\limits^n_{i=1}
\bba_{x_{i}}(\overline{a}_i;b_i) \to 
\bba_{x(x_{1},\ldots,x_{n})}(\overline{a}_1,\ldots,\overline{a}_n;c)
\]
are natural in each $\overline{a}_{i}$, $b_i$ and $c$ as well as the variables $x,x_i$.
\end{enumerate}
\end{proposition}

\proof
The requirements force us to define $\bba_{x}(a_{1},\ldots,a_{n};-)$ using
the substitution map
\[ \xymatrix @R0pc {
\bba_{e}(b;c)\x  
\bba_{x}(a_{1},\ldots,a_{n};b) \ar[r] & \bba_{x}(a_{1},\ldots,a_{n};c) \\
 } \]
 and in the other variables using the substitution map
 \[ \xymatrix @R0pc {
\bba_{x}(b_1,\ldots,b_n;c)\x \prod\limits^n_{i=1} 
\bba_{e}(a_{i};b_i) \ar[r] & \bba_{x}(a_{1},\ldots,a_n;c) \\
 }. \qedhere\]

There is a straightforward way to adapt the notion of morphism of
multicategories and 2-cell to our $\ct$-dependent context. A morphism of
$\ct$-multicategories $\bba\to\bbb$ involves an assignment $a\mapsto Fa$
on objects, together with maps 
\[ F\colon \bba_x(a_1,\ldots,a_n;b)\to \bbb_x(Fa_1,\ldots,Fa_n;Fb) \]
which preserve substitution and identities in the obvious sense. 

\begin{notation}
  When we wish to apply a map $F$ to each element of a list
  $\overline{a}$, we write $F\overline{a}$. Thus the maps displayed
  above could be written as 
\[ F\colon \bba_x(\overline{a};b) \to \bbb_x(F\overline{a};Fb). \]
\end{notation}
Given two such morphisms $F$ and $G$, a 2-cell $\phi\colon F\to G$ involves a
morphism $\phi_a\colon Fa\to Ga$ in $\bbb_e$ for each $a\in\bba$,
subject to the naturality condition asserting that the squares 
\[ \xymatrix @C5pc {
\bba_x(\overline{a};b) \ar[r]^{F} \ar[d]_{G} &
\bbb_x(F\overline{a};Fb) \ar[d]^{\bbb_x(F\overline{a};\phi_b)}
\\
\bba_x(G\overline{a};Gb)
\ar[r]_{\bba_x(\phi_{a_1},\ldots,\phi_{a_n};Gb)} & 
\bbb_x(F\overline{a};Gb) } \]
commute.

$\ct$-multicategories, their morphisms, and their  2-cells together form
a 2-category \TMult.

\begin{definition}\label{defn:weaklyrep}
  A $\ct$-multicategory is {\em weakly representable} when each of the
  functors $\bba_x(a_1,\ldots,a_n;-)\colon \ca \to\Set$ is
  representable. 
\end{definition}
Explicitly, this means that for each $x\in \ct_n$ and each
$a_1,\ldots,a_n\in A$ there exists an object $m_{x}(a_{1},\ldots,a_{n}) \in A$ and multimap $$\theta_{x}(\overline{a}) \in \mathbb  A_{x}(\overline{a};m_{x}\overline{a})$$ with the property that the induced function
$$- \circ_{1} \theta_{x}(\overline{a})\colon\mathbb A_{e}(m_{x}\overline{a};b) \to A_{x}(\overline{a};b)$$ is a
bijection for all $b \in A$. We sometimes call $\theta_x(\overline{a})$ a {\em
  universal multimap} of type $x$. 

\begin{remark}
In the case that $\ct$ is the terminal $\Cat$-operad  there is a single object 
$m(a_1,\ldots,a_n)$ for each $n$. A weakly representable multicategory is {\em
  representable} \cite{Hermida2000Representable} when substitution with $\theta(\overline{a})$ 
induces bijections 
\[ \bba(\overline{b},m\overline{a},\overline{c};d)\to 
\bba(\overline{b},\overline{a},\overline{c};d). \] 
It is possible to make a corresponding definition of
representable $\ct$-multicategory for general $\ct$; we do not do so,
since in the case $\ct=\crr$ corresponding to skew multicategories this is
not the representability condition which captures the notion of skew
monoidal category. In fact for general $\ct$, the only representability
condition we consider is that of weak representability. 
\end{remark}

Returning to the case of a general $\ct$ and a weakly representable
$\ct$-multicategory \bba, by 
Proposition~\ref{prop:categoryA} we have functors 
\[ \bba_{-}(-,\ldots,-;-) \colon \ct_n \x (\ca^n)\op \x \ca  \to \Set \]
which are representable in the last variable, so in the usual way
there is an induced functor 
\[ \ct\op_n \x \ca^n \to \ca, \]
whose action on objects we may write as $(x,a_1,\ldots,a_n)\mapsto
m_x(a_1,\ldots,a_n)$, together with bijections 
\[ \ca(m_x(a_1,\ldots,a_n),b) \cong \bba_x(a_1,\ldots,a_n;b) \] 
natural in all variables $x,a_1,\ldots,a_n,b$. 

These functors $\ct\op_n\x \ca^n\to \ca$ can in turn be thought of as functors 
\[ m \colon \ct\op_n\to [\ca^n,\ca] \]
where $[\ca^n,\ca]$ represents the functor category. 
The category \ca together with the functors $m$
capture much of the structure of the $\ct$-multicategory: the set of
objects, the various multihoms, the identities, and a few special
cases of the substitutions. We shall see in Section~\ref{sect:Talg}
that the  
remaining structure can be understood in terms of colax algebras.

\begin{proposition}\label{prop:unit}
  For a weakly representable $\ct$-multicategory \bba, it is possible to
  choose the corresponding functor $m\colon \ct\op_1\to [\ca,\ca]$ to
  send the unit $e\in \ct\op_1$ to the identity functor. 
\end{proposition}

\proof
The functor $m_e\colon \ca\to\ca$ is characterized by the fact that
it provides representing objects for the Yoneda embedding
$\ca\to[\ca\op,\Set]$, in the sense that there are natural
isomorphisms $\ca(m_{e}(a),b)\cong \ca(a,b)$, so clearly we may take it
to be the identity. 
\endproof

There are various other perspectives on $\ct$-multicategories, not
needed for this paper, but which may help to shed light on the
concept.   We begin with the simplified case of a $\Set$-operad by way of motivation.

\begin{AlternativeViewpoint}
If $\ct$ is a $\Set$-operad we can view it as a $\Cat$-operad
in which each category $\ct_n$ is discrete.  Then a $\ct$-multicategory
amounts to an ordinary multicategory equipped with a multifunctor into $\ct$,
itself viewed as a one object multicategory.

The $\Set$-operad gives rise to a cartesian monad on $\Set$;
one can then define generalised multicategories relative to this cartesian monad -- for instance, see \cite{Leinster-book} -- and these coincide with our $\ct$-multicategories.

Furthermore the monad on $\Set$ extends to a monad $\mathbb \ct$ on the pseudo-double
category of spans, and a $\ct$-multicategory in our sense is then the same as a $\mathbb \ct$-monoid
in the sense of \cite{CruttwellShulman}.
\end{AlternativeViewpoint}

\begin{AlternativeViewpoint}
 Now suppose that $\ct$ is a general \Cat-operad. 
If $\bba$ is a $\ct$-multicategory, there is an associated \Cat-enriched
multicategory $\overline{\bba}$ with the same objects. The multihom
$\overline{\bba}(a_1,\ldots,a_n;b)$ is the category of elements of the
functor $\bba(a_1,\ldots,a_n;b)\colon \ct_n\to\Set$. The substitutions
for $\bba$ induce the necessary substitutions for
$\overline{\bba}$. The projections
$\overline{\bba}(a_1,\ldots,a_n;b)\to \ct_n$ are of course discrete
opfibrations, and they define a \Cat-enriched
multifunctor from $\overline{\bba}$. 

This provides an alternative characterization of $\ct$-multicategories,
as the \Cat-enriched multicategories  in the usual sense, equipped with a \Cat-enriched
multifunctor into $\ct$  which is locally a discrete opfibration. 
Here
when we speak of ``multifunctor into $\ct$'', we are again regarding
the operad $\ct$ as a one-object multicategory. This
in turn makes it clear that one could define $\ct$-multicategories for
any \Cat-enriched multicategory  $\ct$, not just an operad. 

A $\Cat$-operad $\ct$ gives rise to a 2-monad on $\Cat$ which extends to a monad $\mathbb T$ on the pseudo-double category of categories and profunctors; now a $\ct$-multicategory in our sense is the
 same as a $\mathbb T^{ *}$-monoid \cite{CruttwellShulman} with discrete underlying category, where $\ct^{*}$ is the dual operad.
 \end{AlternativeViewpoint}

\begin{AlternativeViewpoint}
There is yet another possible characterization, which is relevant to
what follows: a colax $\ct$-algebra in the monoidal bicategory
$\Mat$ of \Set-valued matrices. An object of $\Mat$ is a set $A$. A
morphism from $A$ to $B$ is an $(A\x B)$-indexed family of sets, and a
2-cell is an $(A\x B)$-indexed family of functions. Morphisms are composed using
the usual formula for matrix multiplication.   

For each set $A$ there is a (pseudo) \Cat-enriched operad $\End(A)$,
with $\End(A)_{ n}=\Mat(A^n,A)$,
and a colax $\ct$-algebra is a colax morphism of operads from $\ct$ to
$\End(A)$.  This involves a functor $\ct_n\to \Mat(A^n,A)$ for each $n$,
sending $x\in \ct_n$ to the family $\bba_x$ whose component indexed by
$(a_1,\ldots,a_n)$ and $b$ is $\bba_x(a_1,\ldots,a_n;b)$. The colax
structure is given by the substitution and identity maps.  
\end{AlternativeViewpoint}

\section{Skew multicategories}\label{sect:Skewmult}

In this section we specialize to a particular \Cat-operad $\crr$, and
define skew multicategories to be $\crr$-multicategories. It turns out
that the (strict) $\crr$-algebras are precisely the skew monoidal categories for
which the associativity maps $\alpha$ and the left unit maps $\lambda$
are identities, although we shall never need to use this fact.

We define $\crr$ explicitly as follows
\[ \crr_n =
\begin{cases}
  \{\ell\} & \text{if $n=0$} \\
  \{\mylam\colon t\to\ell\} & \text{otherwise}
\end{cases}
\]
so that abstractly $\crr_n$ is the arrow category $\two$ for $n>0$ and
$\crr_0$ is the terminal category $1$.

The multiplication 
 $\crr_n\x \crr_{k_1}\x\ldots\x \crr_{k_n}\to \crr_{k_1+\ldots+k_n}$ is defined by 
 \[ x(x_1\ldots,x_n) =
 \begin{cases}
   \myt & \text{if  $x=x_1=\myt$} \\
   \myl & \text{otherwise}
 \end{cases}
\]
and the unit by $\myt\in \crr_1$.  

\begin{remark}
For an object or morphism  $x$ in
$\{\lambda \colon \myt\to\myl\}$ we sometimes write $x_n$ when we are thinking of it as
lying in $\crr_n$.  
\end{remark}

\begin{definition}
  A {\em skew multicategory} is an $\crr$-multicategory. 
\end{definition}

Let us unpack the definition.  To begin with, a skew multicategory involves
\begin{itemize}
\item a set of objects $A$
\item for each $a\in A$ a set $\bba_\myl(~;a)$ of \emph{nullary maps}
\item for each $n>0$, each $a_1,\ldots,a_n\in A$, and each $b\in A$ a
  function
  \begin{equation}\label{eq:comparison}
j_{\overline{a},b}\colon \bba_\myt(a_1,\ldots,a_n;b)\to \bba_\myl(a_1,\ldots,a_n;b) \hspace{0.1cm} .
\end{equation}
\end{itemize}
We sometimes refer to the elements of $\mathbb
A_{t}(a_{1},\ldots,a_{n};b)$ as \emph{tight} $n$-ary multimaps, and to
the elements of $\bba_{\myl}(a_{1},\ldots,a_{n};b)$ as \emph{loose} $n$-ary multimaps.  The functions \eqref{eq:comparison} then allow us to view each tight multimap as a loose multimap. 
\newline{}
On top of this there is further structure: 
\begin{itemize}
\item for each $a\in A$ there is a tight multimap $1_a\in \bba_\myt(a;a)$;
\item  substitution gives us multimaps $g(f_{1},\ldots,f_{n})$,
  which are tight just when $g$ and $f_{1}$ are;  these substitutions,
  moreover, commute with the comparisons viewing tight multimaps as
  loose.  

\end{itemize}
Finally the usual associativity and unit axioms must be satisfied.

In many of our leading examples of skew multicategories the functions \eqref{eq:comparison} are subset inclusions.  In that case we can view skew multicategories as ordinary multicategories equipped with a distinguished class of \emph{tight multimaps}, as we now record.

\begin{proposition}\label{prop:subset}
There is a bijection between 
\begin{enumerate}
\item Skew multicategories  $\mathbb A$ in which each $$j_{\overline{a},b}\colon\mathbb A_{\myt}(a_{1},\ldots,a_{n};b) \to \mathbb A_{\myl}(a_{1},\ldots,a_{n};b)$$ is a \emph{subset inclusion}, and
\item Multicategories $\mathbb A$ together with for each $n>0$
  specified subsets $$\mathbb A_{\myt}(a_{1},\ldots,a_{n};b) \subseteq
  \mathbb A(a_{1},\ldots,a_{n};b)$$ of ``tight" maps containing the
  identities and having the property that a composite multimap
  $g(f_{1},\ldots,f_{n})$ is tight whenever both $g$ and $f_{1}$ are tight.
\end{enumerate}
\end{proposition}

\subsection{Skew multicategories versus ordinary multicategories}
There is a forgetful 2-functor $$U\colon\RMult \to \Mult$$ sending a skew multicategory $\mathbb A$ to the multicategory $\mathbb A_{\myl}$ with the same objects and sets of multimaps $\mathbb A_{\myl}(\overline{a};b)$.  Its underlying ordinary functor has a left adjoint, which views a multicategory as a skew multicategory in which \emph{only the identities are tight}.  $U$ also has a right 2-adjoint which views a multicategory as a skew multicategory in which \emph{all multimorphisms are tight.}  We can identify multicategories with skew multicategories satisfying either condition; the \emph{all multimorphisms are tight} identification has the advantage of extending not only to 1-cells but 2-cells too.

Since $\bba_{\myl}$ is a multicategory it has an underlying category $\ca_{\myl}$ with the same objects as $\bba$ and with morphisms the loose unary maps.  The components $$j_{a,b}\colon\bba_{t}(a;b) \to \bba_{\myl}(a;b)$$ then give the action on morphisms of an identity on objects functor $j\colon\ca \to \ca_{\myl}$, which often has a left adjoint -- see Proposition~\ref{prop:looseclassifier}.

\subsection{2-categorical examples}\label{sect:2cat}

Although no definition of skew multicategory was given in
 \cite{bourko-skew}, various ``2-categorical'' examples were
given there and in \cite{Hyland2002Pseudo}. These are perhaps best introduced via
a simple example. (Each of these examples is, in fact, a $\Cat$-enriched skew multicategory but we will not treat the enrichment here.)

Let $\mathbb {FP}$ denote the multicategory whose objects are
categories equipped with a choice of finite products.  For $n>0$ a
multimap is a functor $F\colon \ca_{1} \times \ldots \times \ca_{n} \to \cb$
preserving products in each variable in the usual up to isomorphism
sense.  A nullary map, an element of  $\mathbb {FP}(~;\cb)$,
 is an object
of $\cb$.  Substitution is defined in the usual way and
the multicategory axioms are routinely verified.  

We declare a multimap $F$ as above to be tight just when it preserves
the given products \emph{strictly in the first variable}; that
is, when each
functor $F(-,a_{2},\ldots,a_{n})\colon\ca_{1} \to \cb$ preserves the given
products strictly.   These tight morphism are easily seen to be
closed under substitution in the first variable, and the complete structure therefore forms
a skew multicategory.  

One can modify this in various ways; for example, there is a
   skew multicategory \bba with the same objects as $\mathbb{FP}$ in which
  $\bba_\ell(\ca_1,\ldots,\ca_n;\cb)$ is given by arbitrary functors $\ca_1\x
  \ldots\x \ca_n\to \cb$, and the tight maps are those which preserve
  finite products (in the usual sense) in the first variable. 

More generally, as proven in \cite{Hyland2002Pseudo}, for any
pseudo-commutative 2-monad $T$ on \Cat, there is a
multicategory whose objects are the (strict) $T$-algebras, and whose
multimaps are the functors $F\colon \ca_1\x \ldots\x \ca_n\to \cb$ equipped
with the structure of an algebra pseudomorphism in each variable
separately, with these $n$ pseudomorphism structures satisfying
certain compatibility conditions.  The tight morphisms, once again
 defined to be those which are  strict in the first variable,
are closed under substitution so that the complete structure forms a skew multicategory.

For instance, one could take $T$ to be the 2-monad for symmetric
strict monoidal categories (also known as permutative categories). 
The corresponding multicategory (of loose maps) was defined in
 \cite[Definition~3.1]{Elmendorf2006Rings}.  Or one could replace
permutative categories by symmetric monoidal categories, braided
monoidal categories, or categories with chosen limits or colimits of
some given class.

\subsection{Skew monoidal categories as skew multicategories}
In Section~\ref{sect:From} we will see that any skew monoidal category $\cc$ gives rise
to a skew multicategory $\bbc$, and that the resulting skew multicategories are precisely
the left representable ones, to which we now turn.
\subsection{Left representability}


By Definition~\ref{defn:weaklyrep}, a skew multicategory $\bba$ is
weakly representable if for 
each pair $x \in \crr_n$ and $\overline{a} \in A^{n}$ there exists an object $m_{x}\overline{a} \in A$ and multimap $$\theta_{x}(\overline{a}) \in \mathbb  A_{x}(\overline{a};m_{x}\overline{a})$$ with the property that the induced function
$$- \circ_{1} \theta_{x}(\overline{a})\colon\mathbb A_{t}(m_{x}\overline{a};b) \to  \bba_{x}(\overline{a};b)$$ is a
bijection for all $b \in A$. 
  
Observe that for all $x \in \crr_n$ we have the equation  $x_{n+m-1} = t_{m} \circ_{1} x_n$.  Therefore a multimap $\theta_{x}(\overline{a})$ as above induces for each $\overline{b} \in A^{m}, c \in A$ a function 
\begin{equation}\label{eq:left}
- \circ_{1} \theta_{x}(\overline{a})\colon \mathbb A_{\myt}(m_{x}\overline{a},\overline{b};c) \to \bba_{x}(\overline{a},\overline{b};c)
\end{equation}
In the case that the above function is invertible for all $\overline{b},c$ as above we say that $\theta_{x}(\overline{a})$ is \emph{left universal}.

\begin{definition}
Let $\mathbb A$ be a weakly representable skew multicategory.  We say that
$\mathbb A$ is \emph{left representable} if the function
\eqref{eq:left} is a bijection for all $x,\overline{a},\overline{b}$ and $c$,
and all universal multimaps $\theta_x(\overline{a})$.
\end{definition}

Let \Skew be the 2-category of skew monoidal categories, (lax)
 monoidal functors, and monoidal natural transformations.
The following theorem, proved in Section~\ref{sect:Skewmon}, is the main result of the present paper.  It is the skew analogue of Theorem 9.8 of \cite{Hermida2000Representable}.

{
\renewcommand{\theproposition}{\ref{thm:leftrep}}
\begin{theorem}
There is a 2-equivalence between the 2-category \Skew  and the full sub 2-category of $\RMult$ consisting of the left representable skew multicategories.
\end{theorem}
\addtocounter{proposition}{-1}
}

In order to obtain a better understanding of left representability, we
isolate two classes of universal multimap.   First, we
refer to a universal multimap $$\theta_{\myl_{0}}\in
\bba_\ell(~;m_{\myl_{0}})$$ as a \emph{nullary map classifier}; and
second, we refer to a universal multimap $$\theta_{t}(a_{1},a_{2}) \in \mathbb A_{t}(a_{1},a_{2};m_{t}(a_{1},a_{2}))$$ as a \emph{tight binary map classifier}.

Given these two classes of multimaps we can construct objects and multimaps $$\theta_{x}(\overline{a}) \in \mathbb  A_{x}(\overline{a};m_{x}\overline{a})$$ for all $x,\overline{a}$ using the inductive formulae
\begin{equation}\label{eq:unary}
m_{t_{1}}(a) = a \textnormal{ and } \theta_{t_{1}}(a) = 1_{a} \in \mathbb A_{t}(a;a)
\end{equation}
\begin{equation}\label{eq:indObject}
m_{x_{n+1}}(a_{1},\ldots,a_{n},a_{n+1}) = m_t(m_{x_{n}}(a_{1},\ldots,a_{n}),a_{n+1})
\end{equation}
and
\begin{equation}\label{eq:indMultimap}
\theta_{x_{n+1}}(a_{1},\ldots,a_{n},a_{n+1}) = \theta_t(m_{x}(a_{1},\ldots,a_{n}),a_{n+1}) \circ_{1} \theta_{x_{n}}(a_{1},\ldots,a_{n})
\end{equation}

\begin{proposition}\label{prop:LeftRep}
Let $\mathbb A$ be a skew multicategory.  The following are equivalent:
\begin{enumerate}
\item $\mathbb A$ is left representable; 
\item $\mathbb A$ admits tight binary map classifiers and a nullary map classifier and the multimaps $\theta_{x}(\overline{a})$ constructed from these according to \eqref{eq:unary}, \eqref{eq:indObject} and \eqref{eq:indMultimap} are universal;
\item $\mathbb A$ admits tight binary map classifiers and a nullary map classifier and these are left universal;
\item $\mathbb A$ is weakly representable and for all $b,c \in A$ the functions
\begin{equation*}
- \circ_{1} \theta_{x}(\overline{a})\colon \mathbb A_{\myt}(m_{x}\overline{a},b;c) \to \mathbb A_{x}(\overline{a},b;c)
\end{equation*}
are invertible.
\end{enumerate}
\end{proposition}
\begin{proof}
By definition of left representability $(1 \implies 3,4)$.  

Suppose that $\mathbb A$ admits tight binary and nullary map classifiers and consider the multimaps $\theta_{x}(\overline{a})$ constructed as in \eqref{eq:unary}, \eqref{eq:indObject} and \eqref{eq:indMultimap}.  Associativity then gives a commutative triangle
\begin{equation}
\xymatrix{
\mathbb A_{ \myt}(m_{t}(m_{x}(\overline{a}),b),\overline{c};d)\ar[d]_{- \circ_{1}\theta_{t}(m_{x}(\overline{a}),b)} \ar[drr]^{- \circ_{1} \theta_{x}(\overline{a},b)} \\
\mathbb A_{ \myt}(m_{x}(\overline{a}),b,\overline{c};d) \ar[rr]_{-\circ_{1} \theta_{x}(\overline{a})} && \mathbb A_{x}(\overline{a},b,\overline{c};d)
}
\end{equation}
for all tuples $\overline{a},\overline{c}$ and objects $b,d$.\newline{}
Let us show that $(2 \implies 1)$.  We must show that $-\circ_{1}
\theta_{x}(\overline{a})\colon  \bba_{x}(m_{x}(\overline{a}),\overline{c};d) \to
\mathbb A_{x}(\overline{a},\overline{c};d)$ is invertible for all for
all $x,\overline{a}$ and tuples $\overline{c}$ of length $n$.  The
case $n=0$ is our hypothesis.  For the inductive step, consider
$(b,\overline{c})$ as a generic $(n+1)$-tuple; we must show that the
horizontal leg above is invertible.  But both the vertical and
diagonal cases are invertible by the inductive hypothesis; hence the
horizontal leg is  so too. 

Let us prove that $(4 \implies 2)$.  By assumption we certainly have
tight binary and nullary map classifiers.  Furthermore the identities
$1\colon a \to a$ of \eqref{eq:unary} are always $t_{1}$-universal.
It remains to prove that if $(4)$ holds then the inductive
constructions of \eqref{eq:indObject} and \eqref{eq:indMultimap}
preserve universality.  That is, we must prove that the composite 
\[ \xymatrix{
\bba_{ \myt} (m_t(m_x(\overline{a}),a_{n+1});b)
\ar[d]_{-\circ_1\theta_{t}(m_x(\overline{a}),a_{n+1})} \\
\bba_{ \myt}(m_x(\overline{a}),a_{n+1};b)
\ar[r]_-{-\circ_1\theta_{x}(\overline{a})} & 
\bba_x(\overline{a},a_{n+1};b) } \]
is invertible. The horizontal map is invertible by universality, and
the vertical is invertible by universality and the assumption in (4).

Finally we prove that $(3 \implies 2)$.  For this the basic and
evident observation is that if $f \in \mathbb A_{x}(\overline{a};b)$
and $g \in \mathbb A_{t}(b,c;d)$ are both left universal then so is $g
\circ_{1} f \in \bba_{x}(\overline{a},b,c;d)$.  Now the identities
$1\colon a \to a$ of \eqref{eq:unary} are always left universal as are the given multimaps in the nullary and binary case by assumption; since the maps $\theta_{x}(\overline{a})$ are obtained from these by composition as above these are always left universal too.
\end{proof}

\subsection{Closed skew multicategories with unit}

\begin{definition}
A skew multicategory $\mathbb A$ is said to be \emph{closed} if for all $b, c \in \mathbb A$ there exists an object $[b,c]$ and tight multimap $e_{b,c}\in \bba_{t}([b,c],b; c)$ 
with the universal property that the induced function
\begin{equation}\label{eq:multi0}
{
e_{b,c} \circ_{1} -\colon\mathbb A_{x}(a_{1}, \ldots, a_{n};[b,c]) \to \mathbb A_{x}(a_{1}, \ldots, a_{n},b;c)}
\end{equation}
is a bijection for all $a_{1},\ldots,a_{n} \in A$ and $x \in \crr_n$.
\end{definition}

If a closed skew multicategory admits a nullary map classifier, then we call it a \emph{closed skew multicategory with unit.}

Let \SkewC be the 2-category of skew closed categories
\cite{skewclosed}, closed functors and closed  natural 
transformations.  In Section~\ref{sect:Skewmon}  we prove the following theorem.
{
\renewcommand{\theproposition}{\ref{thm:SkewC}}
\begin{theorem}  
There is a 2-equivalence between the 2-category \SkewC and the full sub 2-category of $\RMult$ consisting of the closed skew multicategories with unit.
\end{theorem}
\addtocounter{proposition}{-1}
}

We will need one element of the above correspondence in the next section.  By ~\eqref{eq:multi0} we have isomorphisms
\begin{equation*}
{
{e_{b,c} \circ_{1} -} \colon\mathbb A_{t}(a;[b,c]) \to \mathbb A_{t}(a,b;c)}
\end{equation*}
and these are natural in $a$.  By Proposition~\ref{prop:categoryA} the right hand side is a functor $(\ca^{2})^{op} \times \ca \to \Set$ whence by Yoneda the objects $[b,c]$ extend uniquely to a functor $[-,-]\colon\ca^{op} \times \ca \to \ca$ for which the isomorphisms are also natural in $b$ and $c$.  In particular, for each $b \in A$ we have a functor $[b,-]\colon\ca \to \ca$.

\subsection{More on left representability and closedness}

Combining the above cases, let \SkewMC be the 2-category of closed skew monoidal categories, lax monoidal functors and monoidal transformations.  In Section~\ref{sect:Skewmon} we prove:
{
\renewcommand{\theproposition}{\ref{thm:leftrepclosed}}
\begin{theorem}  
The 2-equivalence of Theorem~\ref{thm:leftrep} restricts to a
2-equivalence between the 2-category  \SkewMC of \emph{closed} skew
monoidal categories and the full sub-2-category of \RMult consisting
of the left representable \emph{closed} skew multicategories.
\end{theorem}
\addtocounter{proposition}{-1}
}

In general left representability is a stronger condition than weak
representability, but if the skew multicategory is closed then the two
notions coincide, as the following result shows.  
\begin{proposition}\label{prop:closedleftrep}
 For a closed skew multicategory  $\mathbb A$, the following are
equivalent:  
\begin{enumerate}
\item $\mathbb A$ is left representable;
\item $\mathbb A$ is weakly representable;
\item $\mathbb A$ admits a nullary map classifier and tight binary map classifiers;
\item  $\mathbb A$ admits a nullary map classifier and each functor $[b,-]\colon\ca \to \ca$ has a left adjoint.
\end{enumerate}
\end{proposition}

\begin{proof}
We will show that in a closed skew multicategory any universal multimap $\theta_{x}(\overline{a})$ is left universal.  The equivalence of (1) and (2) is then immediate, whilst the equivalence of (1) and (3) then follows from Proposition~\ref{prop:LeftRep}.
We must prove that for all tuples $\overline{b}$ the function
\begin{equation*}
- \circ_{1} \theta_{x}(\overline{a})\colon \mathbb A_{\myt}(m_{x}\overline{a},\overline{b};c) \to \mathbb A_{x}(\overline{a},\overline{b};c)
\end{equation*}
is invertible for all $c$.  We argue by induction on the length $n$ of $\overline{b}$.  The case $n=0$ is assumed.  Observe that the following diagram commutes.
\begin{equation*}
\xymatrix{
\mathbb A_{t}(m_{x}\overline{a},\overline{b};[b^{\prime},c]) \ar[d]_{e_{b^{\prime},c} \circ_{1} -} \ar[rr]^{- \circ_{1} \theta_{x}(\overline{a})} && \mathbb A_{x}(\overline{a},\overline{b};[b^{\prime},c])
 \ar[d]^{e_{b^{\prime},c} \circ_{1} -} \\
\mathbb A_{t}(m_{x}\overline{a},\overline{b},b^{\prime};c) \ar[rr]^{- \circ_{1} \theta_{x}(\overline{a})} && \mathbb A_{x}(\overline{a},\overline{b},b^{\prime};c)
}
\end{equation*}

Therefore the invertibility of the vertical morphisms (by closedness) and the top horizontal morphism (by induction) ensure the invertibility of the bottom horizontal morphism, as required.

To prove (3 $\iff$ 4) it suffices to show that each $[b,-]$ admits a
left adjoint if and only if $\bba$ admits tight binary map
classifiers.  The first condition is equivalent to asking for a
functor $\otimes\colon\ca^{2} \to \ca$ and isomorphisms $\ca(a \otimes
b,c) \cong \ca(a,[b,c])$ natural in each variable.  The second is
equivalent to asking for a functor $m_{t}\colon\ca^{2} \to \ca$ and isomorphisms  $\ca(m_{t}(a,b),c) \cong \bba_{t}(a,b;c)$ natural in each variable.  Since we have isomorphisms $\ca(a,[b,c]) \cong \bba_{t}(a,b;c)$ natural in each variable the result follows.
\end{proof}

\begin{examples}
Each of the 2-categorical skew multicategories described in \ref{sect:2cat} is associated to an \emph{accessible }pseudocommutative 2-monad on $\Cat$.  Such skew multicategories are both left representable and closed.

The results required to establish these claims are contained in
Section 6 of \cite{bourko-skew}.  Briefly, closedness goes back to
Theorem 11 of \cite{Hyland2002Pseudo}, the nullary map classifier is
the free $T$-algebra on $1$ whilst Proposition 6.3
of \cite{bourko-skew} establishes that each $[A,-]$ has a left adjoint.  Accordingly such  skew multicategories are left representable by Proposition~\ref{prop:closedleftrep} above.
\end{examples}

\section{Colax $\ct$-algebras and $\ct$-multicategories}\label{sect:Talg}

In this section we define colax $\ct$-algebras for a $\Cat$-operad $\ct$.  
\begin{definition}
A colax $\ct$-algebra is a category $\ca$ together with
\begin{itemize}
\item functors $m_{n}\colon\ct_n \times \ca^{n} \to \ca$ whose value at $(x,(a_{1},\ldots,a_{n}))$ we denote by $m_{x}(a_1,\ldots,a_n)$;
\item morphisms $p_{a}\colon m_{e}(a) \to a$ natural in $a$;
\item substitution maps
\begin{equation}\label{eq:subColax}
\xymatrix{
m_{x(x_{1},\ldots,x_{n})}(\overline{a}_{1}, \ldots, \overline{a}_{n}) \ar[rr]^-{\Gamma_{x_{1},\ldots,x_{n},x}} && m_{x}(m_{x_{1}}(\overline{a}_{1}),\ldots, m_{x_{n}}(\overline{a}_{n}))
}
\end{equation}
natural in all variables $x,x_{i},\overline{a}_{i}$
\end{itemize}
satisfying the associativity and identity axioms which are the natural
``\ct-typed" analogues of those for colax monoidal categories: see
\cite[Definition~3.1.1]{Leinster-book}, for example, for the dual
case. 
\end{definition}

\begin{example}
  If $\ct$ is the terminal \Cat-operad \tmop with $\tmop_n=1$ for all $n$, a
  colax-$\ct$-algebra is a colax monoidal category. 
\end{example}

\begin{definition}
A colax $\ct$-algebra is said to be \emph{normal} if the morphisms
$p_{a}\colon m_{e}(a) \to a$ are identities.  
\end{definition}

\begin{remark}
We will be interested primarily in normal colax $\ct$-algebras, since these are the ones corresponding to $\ct$-multicategories.  In the context of a normal colax $\ct$-algebra we obtain substitution maps
\begin{equation*}
\xymatrix{
m_{x \circ_{i} y}(a_{1},\ldots,a_{i-1},\overline{b},a_{i+1},\ldots,a_{n}) \ar[r]^-{\Gamma_{y,i,x}} & m_{x}(a_{1},\ldots,a_{i-1},m_{y}(\overline{b}),a_{i+1},\ldots,a_{n})
}
\end{equation*}
as a special case of \eqref{eq:subColax} on setting $x_{j}=1$ for $j \neq i$ and $x=y$.
\end{remark}

A lax morphism of
colax $\ct$-algebras $A \to B$  involves a functor $F\colon\ca \to \cb$
together with natural families of maps 
\[ \tilde{F}_{x,\overline{a}}\colon m_{x}(Fa_1,\ldots,Fa_n) \to Fm_x(a_1,\ldots,a_n) \]
commuting with substitution and identities in the obvious sense. 

Given two such lax morphisms $F$ and $G$,  a 2-cell $\phi\colon F \to G$ is a natural transformation
with the property that the square
\begin{equation*}
\xymatrix{
m_{x}(Fa_{1},\ldots,Fa_{n}) \ar[d]_{m_{x}(\phi_{a_{1}},\ldots,\phi_{a_{n}})} \ar[r]^{\tilde{F}} & Fm_{x}(a_{1},\ldots,a_{n}) \ar[d]^{\phi_{m_{x}(a_{1},\ldots,a_{n})}} \\
m_{x}(Ga_{1},\ldots,Ga_{n}) \ar[r]_{\tilde{G}} & Gm_{x}(a_{1},\ldots,a_{n})
}
\end{equation*}

Normal colax $\ct$-algebras and their morphisms and 2-cells together form
a 2-category $\ncnotdual_{\myl}$. 

Recall that $\ct^*$ denotes the operad obtained from $\ct$ by replacing
each $\ct_n$ by $\ct\op_n$.  

\begin{theorem}\label{thm:T}
There is a fully faithful 2-functor $\nc_{\myl}\to \TMult$ whose
essential image consists of the weakly representable
$\ct$-multicategories. 
\end{theorem}

\proof
We shall go through the structure involved in a weakly representable
$\ct$-multicategory and see that it corresponds to that of a normal
colax $\ct^*$-algebra, and that this correspondence respects the various
notions of morphism and 2-cell.

First of all, as observed in Proposition~\ref{prop:categoryA} above, a
$\ct$-multicategory \bba determines a category $\ca$ with the same
objects, and functors 
\[ \ct_n\x (\ca^n)\op\x \ca \to \Set \]
and so, in the representable case, functors
\[ \ct\op_n\x \ca^n\to \ca. \]
By Proposition~\ref{prop:unit}, we may take $m_e\colon \ca\to\ca$ to
be the identity. 

Conversely, given such functors we may define the $\ct$-multihoms
\[ \bba_x(a_1,\ldots,a_n;b) = \ca(m_x(a_1,\ldots,a_n),b).\]
So far this accounts for the objects, the multihoms, the identities, and the
substitutions of the form $g(f_1,\ldots,f_n)$ and $h(g)$ where $g\in
\bba_e(b;b)$ for some $b$. This structure satisfies the identity laws
as well as associativity of substitution, as far as it is defined. 

Next we turn to the general form of substitution. 
In our current weakly representable setting, this takes the form of
maps 
\[ \ca(m_{x}(\overline{b}),c) \x \prod\limits^n_{i=1} \ca(m_{x_{i}}(\overline{a}_{i}),b_{i}) \to \ca(m_{x(x_{1},\ldots,x_{n})}(\overline{a}_{1},\ldots,\overline{a}_{n}),c) \]
natural in all variables. By naturality in $(b_{1},\ldots,b_{n})$ and the Yoneda
lemma that amounts to giving natural maps 
\[ \ca(m_{x}(m_{x_{1}}(\overline{a}_{1}),\ldots, m_{x_{n}}(\overline{a}_{n})),c) \to \ca(m_{x(x_{1},\ldots,x_{n})}(\overline{a}_{1}, \ldots, \overline{a}_{n}),c) \]
and now by naturality in $c$ and the Yoneda lemma this amounts to
giving natural maps
\[
m_{x(x_{1},\ldots,x_{n})}(\overline{a}_{1}, \ldots, \overline{a}_{n}) \to m_{x}(m_{x_{1}}(\overline{a}_{1}),\ldots, m_{x_{n}}(\overline{a}_{n}))
\]
or, in other words, to a natural transformation $\tilde{m}$ as in
\eqref{eq:subColax}.

The coassociativity condition for $\tilde{m}$ is equivalent to
associativity of substitution. 

This defines a bijective correspondence between normal colax
$\ct^*$-algebras and weakly representable $\ct$-multicategories, with chosen
representations, including the canonical choice of representatives for
the $\bba_e(a;-)$.

Now suppose that $\bba$ and $\bbb$ are weakly representable
$\ct$-multicategories corresponding to normal colax $\ct^*$-algebras \ca and
\cb. 

What is needed to give a morphism $\bba\to\bbb$ of
$\ct$-multicategories? First of all there is an assignment $a\mapsto Fa$ on
objects. Next there are maps $F_x\colon \bba_x(a_1,\ldots,a_n;b)\to
\bbb_x(Fa_1,\ldots,Fa_n;Fb)$; in particular, there are maps
$\bba_e(a;b)\to \bbb_e(Fa;Fb)$ defining a functor $F\colon
\ca\to\cb$. The remaining $F_x$ correspond, via Yoneda, to maps 
\[ \xymatrix{  m_x(Fa_1,\ldots,Fa_n) \ar[r]^{\tilde{F}} &  Fm_x(a_1,\ldots,a_n). }\]
The functoriality condition on the $F_x$ corresponds to the
associativity condition for the $\tilde{F}$ to define a morphism of
normal colax $\ct^*$-algebras. 

The case of 2-cells follows similarly using Yoneda once again. 
\endproof

\subsection{The skew case}

By Theorem~\ref{thm:T} we know that the 2-category of normal colax
$\crr^*$-algebras is equivalent to the 2-category of weakly representable
$\crr$-multicategories.  From now on, we shall be more interested in $\crr^*$
 than $\crr$; it is therefore convenient to rename it $\cl$. In particular $\cl_0=\{l_{0}\}$ and 
$\cl_n=\{\lambda_{n}\colon l_{n} \to t_{n}\}$ for $n > 0$.

The following is the colax $\cl$-algebra version of left
representability; it appeared in \cite{Fsk2}. 

\begin{definition}
An $\LBC$-algebra is a normal colax $\cl$-algebra for which the maps $\Gamma_{\myt,1,x_n}\colon m_x(a_1,\ldots,a_{n+1})\to
  m_{\myt}(m_x(a_1,\ldots,a_n),a_{n+1})$ are identities for all $n$ and
  all $x\in \cl_n$.
\end{definition}
 
In the name \LBC-algebra, the ``LB'' stands for left-bracketed and the
``C'' for colax; see \cite{Fsk2}.  
We define $\LBCAlgc$ to be the full sub-2-category of $\ncl_{\myl}$ consisting
of the \LBC-algebras.

\begin{theorem}\label{thm:rep}
The 2-equivalence $\ncl_{\myl} \simeq \RMult$ restricts to a 2-equivalence
between $\LBCAlgc$ and the full sub-2-category of $\RMult$ consisting of the left
representable skew multicategories.
\end{theorem}

\proof
Given an \LBC-algebra $\ca$, the corresponding weakly representable
$\crr$-multicategory $\mathbb A$  is related to it as in  the
following equation.  \begin{equation*}
\xymatrix{
\ca(m_{t}(m_{x}(a_{1},\ldots, a_{n}),a_{n+1}),b)\ar@{=}[d] \ar[rr]^-{\ca(\Gamma_{\myt_2,1,x_n},b)} && \ca(m_{t}(m_{x}(a_{1},\ldots,a_{n+1}),b) \ar@{=}[d] \\
\mathbb A(m_{x}(a_{1},\ldots, a_{n}),a_{n+1};b) \ar[rr]_-{\theta_{x}(a_{1},\ldots,a_{n}) \circ_{1} -} && \mathbb A(a_{1},\ldots,a_{n+1};b)
}
\end{equation*}
Therefore by Proposition~\ref{prop:LeftRep}(4) $\mathbb A$ is left
representable.  In the other direction, it suffices to show that each
left representable skew multicategory arises -- up to isomorphism in
$\RMult$ -- from an \LBC-algebra.  By
Proposition~\ref{prop:LeftRep}(2) we can equip $\mathbb A$ with a
choice of multimap classifiers satisfying $m_{t}(m_{x}(a_{1},\ldots,
a_{n}), a_{n+1})=m_{x}(a_{1},\ldots,a_{n+1})$ and
with a similar equation for the universal multimaps, and with
$m_{t_{1}}(a)=a$.  With this choice, the corresponding colax
$\cl$-algebra satisfies  the \LBC\ property, as required. 
\endproof

\section{Skew multicategories versus skew monoidal  categories
  and skew closed categories}\label{sect:Skewmon}

We are now in a position to prove our first main result, which
combines Theorem~\ref{thm:rep} above with Theorem 7.8 of the companion paper \cite{Fsk2}.

\begin{theorem}\label{thm:leftrep}
There is a 2-equivalence between the 2-category \Skew and the full sub 2-category of $\RMult$ consisting of the left representable skew multicategories.
\end{theorem}
\begin{proof}
By \cite[Theorem~7.8]{Fsk2} we have a 2-equivalence $\Skew \simeq \LBCAlgc$.  By Theorem~\ref{thm:leftrep} above we have a 2-equivalence between $\LBCAlgc$ and the full sub-2-category of $\RMult$ consisting of the left representable skew multicategories.  Combining these gives the result.
\end{proof}

We now break down the above processes to give a direct description of the relationship between skew monoidal categories and left representable skew multicategories.

\subsection{From a skew monoidal category to a left representable skew multicategory}\label{sect:From}

Let $\cc$ be a skew monoidal category with unit $i$.  We write $a_{1}\ldots a_{n}$ for the left bracketed tensor product in $\cc$; thus $a_{1}a_{2}$ denotes the usual tensor product with the formula $a_{1}\ldots a_{n}a_{n+1} = (a_{1}\ldots a_{n})a_{n+1}$ determining the higher bracketings.

The corresponding \LBC-algebra structure on $\cc$ has 
\begin{align*}
  m_{\myl}(-) &=i\\
  m_{t}(a_{1},\ldots,a_{n}) &=a_{1}\ldots a_{n} \\
  m_{\myl}(a_{1},\ldots,a_{n}) &=ia_{1}\ldots a_{n} 
\end{align*}
with $m_{\lambda}(a_{1},\ldots,a_{n})\colon m_{\myl}(a_{1},\ldots,a_{n}) \to m_{t}(a_{1},\ldots,a_{n})$ given by 
\begin{equation*}
\xymatrix{
ia_{1}\ldots a_{n} \ar[rr]^{\lambda_{a_{1}}a_{2}\ldots a_{n}} && a_{1}\ldots a_{n} & .
}
\end{equation*}

The substitution morphisms 
\[  \xymatrix @C4pc { 
m_{x(x_{1},\ldots,x_{n})}(\overline{a}_{1},\ldots,\overline{a}_{n})
\ar[r]^-{\Gamma_{x,x_{1},\ldots,x_{n}}} &
  m_{x}(m_{x_{1}}(\overline{a}_{1}),\ldots,m_{x_{n}}(\overline{a}_{n}))
}
\]
are the unique natural families definable for each skew monoidal $\cc$ naturally in $\cc$.   These are obtained by repeated applications of the right unit maps $\rho$ followed by applications of associativity maps $\alpha$, each possibly tensored on either side.  For instance $m_{\myl}(a,b,c,d) \to m_{\myl}(m_{t}(a,b),m_{\myl}(c,d))$ is the map given by
\begin{equation*}
\xymatrix @C1.4pc {
(((ia)b)c)d \ar[rr]^-{((i(ab))\rho)d} && (((ia)b)(ic))d  \ar[rr]^{((\alpha(ic))d } && (((i(ab))(ic))d \ar[r]^-{\alpha} & (i(ab))((ic)d).
}
\end{equation*}
See Section 7.3 of the companion paper \cite{Fsk2} for further details on the \LBC-algebra associated to a skew monoidal category.

Accordingly,  the corresponding skew multicategory $\bbc$ has 
\begin{align*}
  \bbc(~;a) &= \cc(i,a) \\
  \bbc_\myt(a_1,\ldots,a_n;b) &= \cc(a_1\ldots a_n,b) \\
  \bbc_\myl(a_1,\ldots,a_n;b) &= \cc(ia_1\ldots a_n,b) 
\end{align*}
with $j_{\overline{a},b}\colon\bbc_\myt(a_1,\ldots,a_n;b)\to \bbc_\myl(a_1,\ldots,a_n;b)$ given by
\begin{equation*}
\xymatrix{
\cc(a_{1}\ldots a_{n},b) \ar[rr]^{- \circ \lambda_{a_{1}}a_{2}\ldots a_{n}} && \cc(ia_{1}\ldots a_{n},b) .
}
\end{equation*}

The substitution morphisms send $$(f,g_{1},\ldots,g_{n}) \in \cc(m_{x}(\overline{b}),c) \x \prod\limits^n_{i=1} \cc(m_{x_{i}}(\overline{a}_{i}),b_{i})$$ to the composite
$f \circ m_{x}(g_{1},\ldots,g_{n}) \circ \Gamma_{x,x_{1},\ldots,x_{n}}(\overline{a}_{1},\ldots,\overline{a}_{n})$.

\subsection{From a left representable skew multicategory to a skew monoidal category}

This construction is more straightforward and we give it directly,
without mentioning the intermediate colax $\cl$-algebra structure explicitly.  Let $\cc$ be a left representable skew multicategory with $\cc$ its underlying category.  

Tight binary multimap classifiers $a \otimes b$ give representations $$\cc(a \otimes b,c) \cong \bbc_{t}(a,b;c)$$
and we write $\theta_{t}(a,b) \in  \bbc_{t}(a,b;a\otimes b)$
for the universal multimap.  Functoriality of the right hand side transports to yield a functor $\otimes\colon\cc^{2} \to \cc$.

The nullary map classifier $i$ is defined by a representation $$\cc(i,a) \cong \bbc_{\myl}(-;a)$$
with universal multimap $\theta_{ \myl_{0}} \in \bbc_{\myl}(-;i)$. The object $i$ plays the role of the unit. 
\begin{itemize}
\item By left representability we have $$\cc((a\otimes b) \otimes c,d) \cong \bbc_{t}(a \otimes b,c;d) \cong \bbc_{t}(a,b,c;d) \hspace{0.3cm} .$$
Taking $d=a \otimes (b \otimes c)$ and $$\theta_{t}(a, b \otimes c) \circ_{2} \theta_{t}(b,c) \in \bbc_{t}(a,b,c;a \otimes (b \otimes c))$$
then gives rise to the associator $\alpha\colon (a \otimes b)\otimes c \to a \otimes (b \otimes c)$.
\item By left representability we have
$$\cc(i \otimes a,b) \cong \bbc_{t}(i,a;b) \cong \bbc_{\myl}(a;b)$$
and now taking $b=a$ and the image of the identity $1_{a}$ under
$\bbc_{t}(a;a) \to \bbc_{\myl}(a;a)$ yields the left unit map
$\lambda\colon i \otimes a \to a$.
\item The right unit map $\rho\colon a \to a \otimes i$ is the composite $\theta_{t}(a,i) \circ_{2} \theta_{\myl_{0}}  \in \bbc_{t}(a;a \otimes i)$.
\end{itemize}

The left unit map admits another interpretation worth mentioning and which follows immediately from its construction above.
\begin{proposition}\label{prop:looseclassifier}
Let $\bbc$ be a left representable skew multicategory.  The identity
on objects functor $j\colon\cc \to \cc_{\myl}$ has a left
adjoint, whose counit is the left unit map $\lambda\colon i \otimes a \to a$ for the corresponding skew monoidal structure.
\end{proposition}
Accordingly $i \otimes a$ classifies loose unary maps.  In particular, in the setting of a pseudo-commutative 2-monad $T$, wherein the identity on objects functor is the inclusion $j\colon\TAlgs \to \TAlg$ viewing strict morphisms as pseudomorphisms, the tensor product $i \otimes a$ is the \emph{pseudomorphism classifier} \cite{Blackwell1989Two-dimensional}.

We can specialise the equivalence of Theorem~\ref{thm:leftrep} as
follows.  We can identify multicategories with skew multicategories
all of whose multimorphisms are tight; thus we can speak of a left
representable multicategory.  Recall that a skew monoidal category
$\cc$ is said to be left normal if the left unit map $\lambda\colon i \otimes a \to a$ is invertible.

\begin{theorem}\label{thm:leftnormal}
There is a 2-equivalence between the 2-categories of left normal skew monoidal categories and of left representable multicategories.
\end{theorem}
\begin{proof}
The 2-functor $\iota\colon \Mult \to \RMult$ viewing multicategories
$\bbc$ as skew multicategories in which all multimaps are tight
exhibits $\Mult$ as a full sub-2-category of $\RMult$.  In such a
$\bbc$ the inclusion $j\colon\cc \to \cc_{\myl}$ is the identity
whereby its left adjoint has invertible counit.  Since in the left
representable case Proposition~\ref{prop:looseclassifier} ensures that
the counit is $\lambda\colon i \otimes a \to a$, we conclude that the associated skew monoidal category is left normal.

In the opposite direction let $\cc$ be left normal skew monoidal.  By their construction in Section~\ref{sect:From} the components $j_{\overline{a},b}$ are the maps
\begin{equation*}
\xymatrix{
\cc(a_{1}\ldots a_{n},b) \ar[rr]^{- \circ \lambda_{a_{1}}a_{2}\ldots a_{n}} && \cc(ia_{1}\ldots a_{n},b) .
}
\end{equation*}
and therefore are invertible whenever $\cc$ is left normal.  Now the
forgetful $U\colon \RMult \to \Mult$ has right adjoint $\iota$ and the unit component of the adjunction is invertible precisely at those $\cc$ with the above property; thus $\cc$ is isomorphic to such a multicategory.  It follows that the 2-equivalence of Theorem~\ref{thm:leftrep} restricts to yield the desired one.
\end{proof}

Thus the skew aspect of a skew multicategory arises from the
(possible) failure of left normality in the corresponding skew
monoidal category. 

\subsection{Skew monoidal closed categories}
Recall  \cite{skewclosed}  that a skew monoidal category
\cc is said to be \emph{closed} if for all $b,c \in \cc$ there 
exist an object $[b,c]$ and morphism $e_{b,c}\colon[b,c] \otimes b\to c$ such that the induced function

$$
e_{b,c} \circ (- \otimes 1_{b})\colon \cc(a,[b,c]) \to \cc(a \otimes b,c)
$$
is a bijection for all $a$.  

\begin{theorem}\label{thm:leftrepclosed} 
The 2-equivalence of Theorem~\ref{thm:leftrep} restricts to a
2-equivalence between the 2-category \SkewMC of closed skew
monoidal categories and the full sub-2-category of \RMult consisting
of the left representable closed skew multicategories.
\end{theorem}
\begin{proof}
We must show that a left representable skew multicategory
$\bbc$ is closed if and only if the corresponding skew monoidal
category $\cc$ is so.  A tight multimap $e_{b,c} \in
\bbc_{t}([b,c],b;c)$ is precisely a morphism $e_{b,c}\colon [b,c] \otimes b \to c \in \cc$.  Expressed in terms of $\bbc$ the family $\{e_{b,c};b,c \in \bbc\}$ exhibits the skew monoidal $\cc$ as closed just when the induced function
\[ e_{b,c} \circ_{1} -\colon\bbc_{t}(a,[b,c]) \to \bbc_{t}(a,b;c) \]
is a bijection for all $a$.  Now this is certainly required for $\bbc$ to be a closed skew multicategory but the full condition asks that the bottom row below
\begin{equation*}
\xymatrix{
\bbc_{t}(m_{x}\overline{a};[b,c]) \ar[d]_{- \circ_{1} \theta_{x}(\overline{a})} \ar[rr]^{e_{b,c} \circ_{1} -} && \mathbb C_{t}(m_{x}(\overline{a}),b;c)
 \ar[d]^{- \circ_{1} \theta_{x}(\overline{a})} \\
\bbc_{x}(\overline{a};[b,c]) \ar[rr]^{e_{b,c} \circ_{1} -} && \mathbb C_{x}(\overline{a},b;c)
}
\end{equation*}
is a bijection for all $x ,\overline{a}$.  By left representability,
however, the universal multimap $\theta_{x}(\overline{a}) \in
\bbc_{x}(\overline{a};m_{x}(\overline{a}))$ induces bijections in the
columns; since the diagram commutes the top row is a bijection just
when the bottom is one, thus \bbc is closed just when \cc is  so.
\end{proof}

A natural class of skew multicategories  consists of those
for which
each $$j_{\overline{a},b}\colon\bbc_{t}(\overline{a},b) \to
\bbc_{\myl}(\overline{a},b)$$ is an inclusion -- for, as noted in
Proposition~\ref{prop:subset}, these are just multicategories equipped
with a subcollection of tight morphisms closed under 
substitution  in the first variable.  By an argument similar to
Theorem~\ref{thm:leftnormal}, the left representable amongst these
correspond to skew monoidal categories for
which $$\lambda_{a_{1}}a_{2}\ldots a_{n}\colon ia_{1}\ldots a_{n} \to a_{1} \ldots a_{n}$$ is an \emph{epimorphism} for all non-empty tuples $\overline{a}$, wherein the above morphism involves left bracketings.  In the closed skew monoidal case this simplifies since each $- \otimes a_{i}$ preserves epimorphisms.  We record the result in that setting, which refines Theorem~\ref{thm:leftrepclosed}.

\begin{theorem}
There is a 2-equivalence between closed skew monoidal categories whose
left unit maps $\lambda\colon i \otimes a \to a$ are epimorphisms and those skew multicategories $\bbc$ with each $j_{\overline{a},b}\colon\bbc_{t}(\overline{a},b) \to \bbc_{\myl}(\overline{a},b)$ is an inclusion.
\end{theorem}

\subsection{Skew closed categories versus closed skew multicategories with unit}

A skew closed category \cite{skewclosed} consists of a
category $\cc$ equipped with a functor $[-,-]\colon \cc^{op} \times \cc \to \cc$ and object $i$ together with natural transformations
\[ \xymatrix @R0pc { 
[b,c] \ar[r]^-{L} &  [[a,b],[a,c]] \\
[i,a] \ar[r]^{I} & a \\
i \ar[r]^{J} & [a,a]
} \]
subject to five axioms \cite{skewclosed}.  If the components $I\colon
[i,a] \to a$ are invertible as well as the functions $\cc(J,1) \circ
[a,-]\colon\cc(a,b) \to \cc(i,[a,b])$, the skew closed category is
said to be \emph{closed}.\begin{footnote}{ The original definition of
    closed category \cite{Eilenberg1966Closed} involved an
    \emph{underlying functor} to $\Set$.   Above we refer to the modified definition of \cite{Laplaza1977Embedding}.}\end{footnote}

Between skew closed categories are \emph{closed functors}, which
involve morphisms $[Fa,Fb] \to F[a,b]$ and $i \to Fi$,  and {\em
  closed transformations}.  All together, these form a 2-category $\SkewC$.

In Theorem 5.1 of of \cite{Manzyuk2012Closed}, Manzyuk established a
correspondence between closed categories and closed multicategories
with unit.  The following theorem, which builds on work of
\cite{bourko-skew}, gives the skew version of Manzyuk's theorem.  Our
argument, which is rather different in character to Manzyuk's,
essentially treats the skew closed case as a special case of the skew
monoidal case and can easily be adapted to give an alternative proof of his result.

\begin{theorem}\label{thm:SkewC}
There is a 2-equivalence between the 2-category \SkewC and the full sub 2-category of $\RMult$ consisting of the closed skew multicategories with unit.
\end{theorem}

\begin{proof}
We give only the core details of the proof.  Our argument will proceed
in three steps:
\begin{enumerate}[Step (I)]
\item skew closed structures on \cc correspond to certain right skew
  monoidal structures on $[\cc,\Set]$;
\item the right skew monoidal structures of (I) correspond to certain
  {\em lax} \crr-algebra structures;
\item the lax \crr-algebra structures of (II) correspond to closed
  skew multicategories with unit, whose underlying category is \cc.
\end{enumerate}

 Step (I)  is due to Street  \cite{skewclosed} and builds on work
of Day \cite{daythesis}.  A (left) skew \emph{promonoidal} structure
on $\cc$ is a left skew pseudomonoid in the monoidal bicategory of
profunctors: such involves structure functors $P\colon\cc^{op} \times
\cc^{op} \times \cc \to \Set$ and $J\colon \cc \to \Set$ plus three
coherence constraints satisfying five equations \cite{skewclosed}.  By
Proposition 22 of \cite{skewclosed} each skew closed structure on $\cc$ determines a skew promonoidal structure with $P(a,b,c)=\cc(a,[b,c])$ and $Ja=\cc(i,a)$.  By the same proposition skew closed structures on $\cc$ can be identified with promonoidal structures for which each $P(-,b,c)\colon\cc^{op} \to \Set$ and $J\colon\cc \to \Set$ are representable.  Now left skew promonoidal structures on \cc correspond to \emph{right} skew monoidal structures on $[\cc,\Set]$ whose tensor product is cocontinuous in each variable: the right skew monoidal structure on $[\cc,\Set]$ has a convolution tensor product given by the left Kan extension $m = lan_{y^{2}}P$
\begin{equation*}
\xymatrix{
[\cc,\Set]^{2} \ar[dr]^{m} \\
(\cc^{op})^{2} \ar[u]^{y^{2}} \ar[r]_{P} & [\cc,\Set] 
}
\end{equation*}
and unit $J$.
Putting this together we conclude that skew closed structures on $\cc$ amount to
\begin{enumerate}
\item
Right skew monoidal structures $([\cc,\Set],m,J)$ such that $m$ is
cocontinuous in each variable, such that $m(y-,yb)(c)\colon\cc^{op}
\to \Set$ is representable, and such that $J\colon\cc \to \Set$ is representable.
\end{enumerate}

By Theorem 7.8 of the companion paper \cite{Fsk2},  left skew monoidal structures on a category correspond to \LBC-algebras -- normal colax $\cl$-algebras whose substitution maps 
\begin{equation}\label{eq:sub}
\xymatrix{
m_{x_{n+1}}(\overline{a},b) \ar[rr]^{\Gamma_{\myt_2,1,x_n}} && m_{t_{2}}(m_{x_n}(\overline{a}),b)
}
\end{equation} 
are identities.  If these are merely isomorphisms rather than identities then, by an easy transport of structure argument, we can produce an isomorphic normal colax structure on $\cc$ satisfying the stricter condition.  Accordingly, left skew monoidal structures equally correspond to \emph{normal colax $\cl$-algebras with \eqref{eq:sub} invertible}.  The dual result is that 
right skew monoidal structures on $[\cc,\Set]$ correspond to normal \emph{lax} $\crr$-algebra structures on $[\cc,\Set]$ for which the substitution maps 
\begin{equation}\label{eq:sub2}
\xymatrix{
 m_{t_{2}}(m_{x_n}(\overline{a}),b) \ar[rr]^{\Gamma_{\myt_2,1,x_n}} && m_{x_{n+1}}(\overline{a},b) 
}
\end{equation} 
are invertible.\begin{footnote}{A lax $\crr$-algebra corresponds to a colax $\cl$-algebra structure on the opposite category -- accordingly its substitution maps point in the opposite direction.}\end{footnote}
The lax structure associated to $([\cc,\Set],m,J)$ has $m_{t_{2}}=m$ and $m_{l_{0}}=J$.  The natural isomorphisms $m_{x_{n+1}}(\overline{a},b) \cong  m_{t_{2}}(m_{x_n}(\overline{a}),b)$ inductively ensure that each $m_{x}\colon[\cc,\Set]^{n} \to [\cc,\Set]$ is cocontinuous in each variable if $m$ is.  Accordingly, right skew monoidal structures as per (1) correspond to
\begin{enumerate}
\item[(2)] Normal lax $\crr$-algebra structures on $[\cc,\Set]$ with $m_x$ cocontinuous in each variable for all $x$ and such that
\begin{enumerate}
\item The functors $m_{t_{2}}(y-,yb)(c)\colon\cc^{op} \to \Set$ 
 and $m_{l_{0}}\colon\cc \to \Set$ are representable.
\item The substitution maps ~\eqref{eq:sub2} are invertible.
\end{enumerate}
\end{enumerate}

Accordingly, it remains 
 to establish a correspondence between closed skew multicategories
 with unit on $\cc$ and structures as in (2) above.  In fact, the core
 of this correspondence holds for a general $\Cat$-operad $\ct$:
$\ct$-multicategory structures on $\cc$ correspond to
normal lax $\ct$-algebra structures on $[\cc,\Set]$ for which $m_x\colon[\cc,\Set]^n \to [\cc,\Set]$ is cocontinuous in each variable for each $x \in \ct_n$. 

 To  see  this, first recall from
Proposition~\ref{prop:categoryA} that for a $\ct$-multicategory
\bbc,  the collections of multimaps extend uniquely to functors 
\begin{equation}\label{eq:extension}
\bbc_{-}(-;-)\colon\ct_n \times (\cc^{n})^{op} \times \cc \to \Set .
\end{equation}
such that $\bbc_{e}(-;-) = \cc(-;-)$ and with respect to which substitution becomes natural in each variable.  This allows us to identify $\ct$-multicategory structures on $\cc$ with multicategories equipped with such extensions.
Given such a $\bbc$  the \emph{convolution} lax $\ct$-algebra structure on $[\cc,\Set]$ has $m_x$ given by the left Kan extension
\begin{equation*}
\xymatrix{
[\cc,\Set]^{n} \ar[dr]^{m_x} \\
(\cc^{op})^{n} \ar[u]^{y^{n}} \ar[r]_{\bbc_x} & [\cc,\Set] & .
}
\end{equation*}
Accordingly $m_x$ is cocontinuous in each variable and is given
by  
\begin{equation*}
m_x(A_1,\ldots,A_n) = \int^{a_{1},\ldots,a_{n} \in \cc}A_{1}a_{1} \times \ldots A_{n}a_{n} \times \bbc_x(a_{1},\ldots,a_{n};-) .
\end{equation*}
We sometimes abbreviate this by  $\int^{\overline{a}}\overline{A}\overline{a} \times \bbc_x(\overline{a};-)$.  Since $y^n$ is fully faithful we have $m_x(ya_1,\ldots,ya_n)\cong \bbc_x(a_{1},\ldots,a_n;-)$.

Now $\bbc_e = y\colon\cc^{op} \to [\cc,\Set]$; hence we may, and do,
set $m_e=1$ so as to obtain normality.  The components $m_f\colon m_x
\to m_y$ for $f\colon x \to y \in \ct_n$ satisfy the obvious formula.
The component of substitution $m_{x}(m_{x1},\ldots,m_{xn}) \to
m_{x(x_{1},\ldots,x_n)}$ at $(\overline{A_1},\ldots,\overline{A_n})$
is the composite
\begin{eqnarray*}
\int^{\overline{b}} \left(\int^{\overline{a_1}}\overline{A_1}\overline{a_{1}} 
 \times \bbc_{x_{1}}(\overline{a_{1}};b_{1}) \times \ldots \times \int^{\overline{a_n}}\overline{A_n}\overline{a_{n}} \times \bbc_{x_{1}}(\overline{a_{n}};b_{n}) \right) \times \bbc_x(\overline{b};-) \\
 \cong \int^{\overline{b},\overline{a_1},\ldots,\overline{a_n}}  \overline{A_1}\overline{a_{1}} 
\times \ldots \times \overline{A_n}\overline{a_{n}}  \times \bbc_{x_{1}}(\overline{a_{1}};b_{1}) \times \ldots \bbc_{x_{1}}(\overline{a_{n}};b_{n}) \times \bbc_x(\overline{b};-) \\
\longrightarrow 
\int^{\overline{a_1},\ldots,\overline{a_n}}  \overline{A_1}\overline{a_{1}} 
\times \ldots \times \overline{A_n}\overline{a_{n}}  \times \bbc_{x(x_1,\ldots,x_n)}(\overline{a_1},\ldots,\overline{a_n};-)
\end{eqnarray*}
whose second component is a coend of substitution maps.  On representables it returns, up to natural isomorphism, the substitution maps for $\bbc$.  The lax $\ct$-algebra axioms are easily verified: since they assert the equality of composite natural transformations between functors cocontinuous in each variable it is enough to check they hold at representables, where they amount to the axioms for a $\ct$-multicategory.  We omit the straightforward converse construction, which is obtained by restriction along powers of the Yoneda embedding.

Finally we specialise to $\ct=\crr$.  By the above analysis we have a correspondence between  $\crr$-multicategory
structure   on $\cc$ and normal lax $\crr$-algebra structure on $[\cc,\Set]$ for which each $m_x$ is cocontinuous in each
variable.  It remains,  then, to prove that it restricts to a correspondence between closed skew multicategories with unit and lax $\crr$-algebras  having the properties (2a) and (2b).

Let $\bbc$ be an $\crr$-multicategory.  We investigate what the conditions (2a) and (2b), interpreted at the associated lax $\crr$-algebra, mean for $\bbc$ itself.  The representability conditions of (2a) simply amount to the existence of a nullary map classifier and objects $[b,c]$ equipped with isomorphisms 
\begin{equation}\label{eq:preclosed}
\bbc_{e}(a,[b,c]) \cong \bbc_{t_{2}}(a,b;c)
\end{equation} natural in $a$.  Letting $e_{b,c} \in \bbc_{t_{2}}([b,c],b)$ denote the unit of the representation, we must show that the maps 
\begin{equation}\label{eq:closed}
\xymatrix{ 
\bbc_{x}(\overline{a},[b,c]) \ar[rr]^-{e_{b,c} \circ_{1} -} &&
\bbc_{x_{n+1}}(\overline{a},b;c) } 
\end{equation} are invertible -- that is, \bbc is closed -- if and only if the associated lax $\ct$-algebra satisfies (2b).

By cocontinuity of $m_x$ in each variable, the substitution maps of (2b) will be invertible in all components just when they are so at representables; that is, just when the map
\begin{equation*}
\xymatrix{
 m_{t_{2}}(m_{x_n}(y\overline{a}),yb) \ar[rr]^{\Gamma_{\myt_2,1,x_n}} && m_{x_{n+1}}(y\overline{a},yb) 
}
\end{equation*} 
is so for all $\overline{a} \in \cc^n $ and $b \in \cc$.  This
map is induced by multicategorical substitution: at $c \in \cc$ it has component
$$
\int^{d, f}\bbc_{x}(\overline{a};d) \times \bbc_{e}(b;
f)
\times \bbc_{t_{2}}(d, f;c) \to
\bbc_{x_{n+1}}(\overline{a},b; c)
$$
which, applying Yoneda to the domain, is isomorphic to
$$
\circ_{1}\colon \int^{d}\bbc_{x}(\overline{a};d) \times \bbc_{t_{2}}(d,b;c) \to \bbc_{x_{n+1}}(\overline{a},b;c)
$$
Accordingly we must show that these last maps are invertible just when
those in \eqref{eq:closed} are.  
Now in the commutative diagram below the left vertical map is a Yoneda isomorphism whilst invertibility of the top horizontal map follows from \eqref{eq:preclosed}. 
\[ \xymatrix{
{}\int^{d}\bbc_{x}(\overline{a};d) \times \bbc_{e}(d;[b,c])  \ar[rr]^{\int^{d}1 \times (e_{b,c} \circ_{1} -)} 
\ar[d]_{\circ_{1}} && {}\int^{d}\bbc_{x}(\overline{a};d) \times \bbc_{t_{2}}(d,b;c) \ar[d]^{\circ_{1}} \\
\bbc_x(\overline a;[b,c]) \ar[rr]^{e_{b,c} \circ_{1} -} && \bbc_x(\overline{a},b;c) } \]
Therefore the right vertical map is invertible if and only if the
bottom horizontal map is so.  
\end{proof}

\bibliographystyle{plain}

\end{document}